\documentclass [10pt,         
 		a4paper,      
 	       ]{scrartcl}    

\usepackage{amsmath,amssymb}
\usepackage{url}
\usepackage{dsfont}
\usepackage{xfrac}
\usepackage{mathtools}
\usepackage{csvsimple}
\usepackage{longtable}

\usepackage[english]{babel}          
\usepackage{fixmath}                 
\usepackage{amsthm}                  

\usepackage {color,graphicx,wrapfig}   
\usepackage{colortbl}
\usepackage {booktabs}                 
\usepackage {multirow}

\usepackage{tikz}
\usetikzlibrary{calc}
\usetikzlibrary{circuits}
\usetikzlibrary{shapes,arrows}

\tikzset{>=stealth}

\usepackage{tikzscale}

\usepackage{pgfplots}
\pgfplotsset{compat=1.8}

\usepackage{adjustbox}

\usepackage{subcaption}

\usepackage{makecell}

\usepackage{siunitx}

\graphicspath{{Pictures/}}

\usepackage{todonotes}
\reversemarginpar

\usepackage{hyperref}
\hypersetup {
  hidelinks = true,
  pdfkeywords = {},    	                 
  pdfproducer = {pdfLatex},       	 
  pdfcreator = {LaTeX}, 		 
  pdfpagelayout = SinglePage,
  pdffitwindow = {true}
}

\newcommand{\R}{\mathds{R}}
\newcommand{\N}{\mathds{N}}

\newcommand{\I}{\ensuremath{\mathcal{I}}}

\renewcommand{\O}{\ensuremath{\mathcal{O}}}
\renewcommand{\P}{\ensuremath{\mathcal{P}}\xspace}
\newcommand{\NP}{\ensuremath{\mathcal{NP}}\xspace}
\newcommand{\T}{^{\top}}

\newcommand{\define}{\coloneqq}

\newcommand{\ie}{i.e.,\xspace}
\newcommand{\eg}{e.g.,\xspace}
\newcommand{\tdtsp}{TDTSP\xspace}

\newcommand{\fatline}{\noalign{\vskip 1mm}\Xhline{2\arrayrulewidth}\noalign{\vskip 1mm}}

\DeclareMathOperator{\opt}{opt}
\DeclareMathOperator{\arr}{arr}
\newcommand{\arrtime}{\ensuremath \theta^{\arr}}

\pgfplotscreateplotcyclelist{moreLineStyles}{
  {solid},
  {densely dotted},
  {loosely dashdotted},
  {densely dashed},
  {loosely dotted},
  {densely dashdotdotted},
  {loosely dashed},
  {looseley dashdotdotted},
  {densely dashdotted}
}

\theoremstyle{plain}
\newtheorem{thm}{Theorem}[section]            	           %
\newtheorem{lem}[thm]{Lemma}         			   
\theoremstyle{definition}
\newtheorem{exmp}[thm]{Example}

\newtheorem{rem}[thm]{Remark}
\theoremstyle{remark}

\makeatletter
\renewcommand*\env@matrix[1][*\c@MaxMatrixCols c]{%
  \hskip -\arraycolsep
  \let\@ifnextchar\new@ifnextchar
  \array{#1}}
\makeatother

\definecolor{darkgreen}{rgb}{0.1,0.6,0.1}
\definecolor{darkblue}{rgb}{0.1,0.1,0.6}
\definecolor{darkred}{rgb}{0.9,0.1,0.1}
\definecolor{lgray}{gray}{0.75}
\definecolor{pink}{rgb}{1.0, 0.4, 0.7}

\usepackage{xspace}

\newcommand{\refsec}[1]{Section~\ref{#1}\xspace}

\newcommand{\set}[1]{\{#1\}}

\definecolor{pink}{rgb}{1.0, 0.4, 0.7}

\renewcommand{\T}{\mathcal{T}}

\begin{document}
\title{Cuts, Primal Heuristics, and Learning to Branch for the Time-Dependent 
Traveling Salesman Problem}
\author{Christoph Hansknecht, Imke Joormann, Sebastian Stiller}

\maketitle

\begin{abstract}
\noindent
  We consider the time-dependent traveling salesman problem
  (\tdtsp), a generalization of the asymmetric traveling
  salesman problem (ATSP) to incorporate time-dependent cost
  functions. In our model, the costs of an arc can change arbitrarily over time
  (and do \emph{not} only dependent on the position in the 
tour). 
The TDTSP turns out to be structurally more difficult than the TSP. 
We prove it is NP-hard and APX-hard even if a generalized version of the 
triangle inequality is satisfied. In particular, we show that even the 
computation of one-trees becomes intractable in the case of time-dependent 
costs. 
  
  We derive two
  IP formulations of the \tdtsp based on time-expansion and propose
  different pricing algorithms to handle the significantly increased
  problem size. We introduce multiple families of cutting planes for
  the \tdtsp as well as different LP-based primal heuristics, a
  propagation method and a branching rule. We conduct computational
  experiments to evaluate the effectiveness of our approaches
  on randomly generated instances. 
  We are able to decrease the optimality gap remaining after one hour of 
computations to about six percent, compared to  a gap of more than forty 
percent obtained by an off-the-shelf IP solver.

  Finally, we carry out a first attempt to learn strong branching decisions for 
the TDTSP. At the current state, this method does not improve the 
running times.
\end{abstract}

\section{Introduction}
\label{section:introduction}

The traveling salesman problem (TSP) is among the best studied
combinatorial optimization problem (see \cite{TSPBook,TSPVariations}
for summaries). Considerable effort has been put into polyhedral
analysis of the problem, development of primal heuristics, and
implementation of branch-and-bound based code. Several generalizations
of the problem have been considered as well, such as the TSP with
time windows \cite{ATSPTimeWindow,SolvingATSPTimeWindow}, or
the class of vehicle routing problems (VRPs) \cite{VRPBook}.

The classical asymmetric TSP is based on the assumption that the
travel time $c_{ij}$ for an arc $(i, j)$ is constant throughout the
traversal of the graph by the optimum tour. While this assumption is
justified when it comes to travel times based on geometric distances,
travel times tend to vary over time in real-world instances (such as
road networks).

Some effort has been made in order to generalize the TSP with respect
to time-dependent travel times. The authors 
of~\cite{TimeDependentTSP,TDTSPTardiness} consider the problem of
minimizing the travel time of a tour where the travel time of an arc
$(i, j)$ depends on the position of $i$ in the tour. Thus, the travel
time of $(i, j)$ is a function $c_{ij}(k)$ ($k = 1, \ldots, n$). This
simplified time-dependent TSP (which we denote by STDTSP) has since
attracted some attention, specifically, the authors 
of~\cite{TimeDependentTSPPoly} conduct a polyhedral study and perform
computational experiments. The STDTSP is solved on a graph
which
consists of $n$ layers of vertices, a tour corresponds then to a
path containing exactly one representative of each vertex.
Note that the STDTSP is closely related to identical machine
scheduling, in particular $P || \sum w_j T_j$, which can
be solve in a similar fashion~\cite{ParallelMachines}.

We further generalize the concept of time-dependent travel times
to the case where the travel time of an arc $(u, v)$ is a
function $c_{uv} : \{0, \ldots, \theta^{\max}\} \to \N$. As a result,
the corresponding instances tend to be much larger than in the
case of position-dependent travel times and particular care has
to be taken in order to provide exact solutions within a reasonable
time. It is however possible to generalize many results from the
STDTSP to the real-time-dependent TSP (\tdtsp).
(see, \eg \refsec{section:complexity}).

\section{Preliminaries}
\label{section:preliminaries}
An instance $(D,c,\theta^{\max})$ of the \tdtsp consists of a complete
directed graph $D=(V, A)$ with $n$ vertices ($V=\{1, \ldots, n\}$), a
time-horizon $\theta^{\max}$, and time-dependent travel times
$c_{a}: \Theta \to \N$ for $a \in A$, where $\Theta \define \{0, \ldots,
\theta^{\max}\}$ is
a set of points in time.  The vertex $s\define 1$ is defined as the source
vertex.
For each sequence of arcs $(a_1, \ldots, a_k)$ with $a_k = (u_k, v_k)$ and
$v_k = u_{k + 1}$ we can recursively define an arrival time
\begin{equation}
  \arrtime(a_1, \ldots, a_k) \define
  \begin{cases}
    c_{u_1, v_1}(0), & \text{ if } k = 1, \\
    \arrtime(a_1, \ldots, a_{k - 1}) + c_{u_k, v_k}(\arrtime(a_1, \ldots, a_{k -
1})), & \text{ else}.\\
  \end{cases}
\end{equation}
The asymmetric \tdtsp asks for a tour $T = (a_1, \ldots, a_n)$ which
minimizes the arrival time $\arrtime(a_1, \ldots, a_n)$.

We will consider several special cases of travel time functions which
play an important role in time-dependent versions of combinatorial problems:
\begin{enumerate}
\item
  Several well-known results (\eg \cite{DoubleTree,Christofides})
  state that the symmetric version of
  the TSP can be approximated in case of \emph{metric} cost coefficients,
  i.e. cost coefficients satisfying the triangle inequality.
  The definition of the triangle inequality can be easily generalized
  to the time-dependent case. Formally, a set of travel time functions
  satisfies the \emph{time-dependent triangle inequality}
  iff for each $u, v, w \in V$, $\theta \in \Theta$ with
  $\theta + c_{uv}(\theta) \leq \theta^{\max}$, it holds that
  \begin{equation}
    \label{eq:time_dependent_triangle}
    \theta + c_{uw}(\theta) \leq
    \theta + c_{uv}(\theta) + c_{vw}(\theta + c_{uv}(\theta)).
  \end{equation}
\item
  Another property of time-dependent cost functions goes by the name
  of FIFO (first-in-first-out). A function $f : \N \to \N$ satisfies the
  FIFO-property iff
  \begin{equation}
    \theta + f(\theta) \leq \theta' + f(\theta') \quad \forall\,
    \theta, \theta' \in \N,\; \theta \leq \theta'.
  \end{equation}
  The FIFO property implies that is is never advisable to wait at a
  certain vertex to decrease the arrival time at a destination.
  If the FIFO property is satisfied for each time-dependent cost function,
  then shortest paths with respect to time-dependent costs can be computed
  efficiently using a variant of Dijkstra's
  algorithm~\cite{TimeDependentDijkstra,TimeDependentRoutePlanning}.
\end{enumerate}

\section{Complexity}
\label{section:complexity}
As a generalization of the well-known ATSP, the \tdtsp is $\NP$-hard
itself. What is more, there exists no $\alpha$-approximation for any
$\alpha \geq 1$ for the general ATSP~\cite{Bibel}. On the other hand,
approximation algorithms are known for the metric variant of the
ATSP. Unfortunately, such algorithms don't exist in the case of the
\tdtsp:
\begin{thm}
  \label{thm:inapx}
  There is no $\alpha$-approximation algorithm
  for any $\alpha > 1$ for the \tdtsp
  unless $\P = \NP$. This is the case even if the time-dependent
  triangle inequality is satisfied.
\end{thm}

\begin{proof}
  Suppose there exists an $\alpha$-approximation algorithm $A$ for the \tdtsp
  for a fixed value of $\alpha \geq 1$. We show that algorithm
  $A$ could be used to solve the \emph{Hamiltonian cycle} problem on
  an undirected graph $G=(V, E)$. To this end, let $D=(V, A)$ be the
  bidirected complete graph with costs
  \begin{equation*}
    c_{uv} \define
    \begin{cases}
      1, & \text{ if } \{u, v\} \in E \\
      2, & \text{ otherwise}.
    \end{cases}
  \end{equation*}
  Note that $G$ is Hamiltonian iff $D$ contains a tour with costs of
  at most $n$.  Consider the time-expansion of $D$ given by
  $\theta^{\max} \define \alpha n$ and the following time-dependent
  cost functions (satisfying the time-dependent triangle inequality):
  \begin{equation*}
    c_{uv}(\theta) \define
    \begin{cases}
      c_{uv}, & \text{ if } \theta \leq n \\
      \alpha n + 1, & \text{ otherwise}.
    \end{cases}
  \end{equation*}
  We apply $A$ to the instance $(D, c, \theta^{\max})$. It the
  resulting tour $T$ has $\arrtime(T) \leq n$, it must correspond to
  Hamiltonian cycle in $G$. Otherwise we know that $\arrtime(T) >
  \alpha n$, since $T$ must contain at least one arc $(u, v)$
  such that $T$ arrives at $u$ at a time $\geq n$. Since $A$ is
  an $\alpha$-approximation, the optimal tour $T_{\opt}$ has
  $\arrtime(T_{\opt}) > n$ and $G$ is not Hamiltonian.

\end{proof}

\begin{rem}[Dynamic Programming]
  It is well-known that the (asymmetric) TSP can be solved by using
  a dynamic programming approach: Let $C(S, v)$ be the
  smallest cost of an $(s,v)$-path consisting of the vertices $S \subseteq V$
  with $s, v \in S$. Then $C(S, v)$ satisfies the following
  relations:
  \begin{equation}
    \begin{aligned}
      C(\{s, v\}, v) &= c_{sv} && \forall v \in V, v \neq s \\
      C(S, v) &= \min_{\substack{u \in S\\ u \neq s, v}}
      C(S \setminus \{v\}, u) + c_{uv} && \forall S \subseteq V, v \in S. \\
    \end{aligned}
  \end{equation}
  The cost of an optimal tour is then given by
  $\min_{v \neq s} C(V, v) + c(v, s)$ and can be computed
  in $\O(2^{n}\cdot n^{2})$.
  If a given \tdtsp instance satisfies the FIFO property, these
  relations can be generalized to incorporate time-dependent costs:
  \begin{equation}
    \begin{aligned}
      C(\{s, v\}, v) &= c_{sv}(0) && \forall v \in V, v \neq s \\
      C(S, v) &= \min_{\substack{u \in S\\ u \neq s, v}}
      C(S \setminus \{v\}, u) + c_{uv}(C(S \setminus \{v\}, u))
      && \forall S \subseteq V, v \in S. \\
    \end{aligned}
  \end{equation}
  Note that the complexity is the same as in the case of an ATSP.
  This is due to the fact that the FIFO property ensures that only the
  shortest path for fixed $S, v$ needs to considered for subsequent
  computations. Without the FIFO property it becomes necessary to
  consider an $(s, v)$-paths for each $\theta \in \T(v)$ during the
  computations.
\end{rem}

\subsection{Approximation for Special Cases}

While the \tdtsp problem is relatively hard by itself, some results regarding
approximations can be preserved in the case where the time-dependent cost
functions
are of low variance.

\begin{thm}
  Let $\lambda \geq 1$  such that for all $u, v \in V$, $\theta ,\theta' \in
\{0, \ldots, T\}$
  it holds that
  \begin{equation}
    c_{uv} (\theta) \leq \lambda c_{uv}(\theta').
  \end{equation}
  Then, any $\alpha$-approximation of the TSP yields a $(\alpha
\lambda)$-approximation of
  the \tdtsp.
\end{thm}

\begin{proof}
  Let $c : A \to \N$ be defined as
  \begin{equation}
    c_{uv} \define \min_{\theta \in \{0, \ldots, T\}} c_{uv}(\theta).
  \end{equation}
  This implies that $c_{uv} \leq c_{uv}(\theta) \leq \lambda c_{uv}$ for all
$\theta$.
  Let $T_{\opt}$, $T_{h}$ be the optimal and $\alpha$-approximate tour with
  respect to the costs $c$ and $T_{\opt, t}$ be the optimal tour.  We have that
  \begin{equation}
    \begin{aligned}
      \arrtime(T_h) \leq \lambda \cdot c(T_h) & \leq (\alpha \lambda) \cdot
c(T_{\opt}) \\
      & \leq (\alpha \lambda) \cdot c(T_{\opt, t}) \\
      & = (\alpha \lambda) \cdot \arrtime(T_{\opt, t}) \\
    \end{aligned}
  \end{equation}
\end{proof}
Note that since the ATSP in general is inapproximable in general, further
assumptions,
such as a metric lower bound $c_{uv}$, are still necessary to obtain an
approximation.

\subsection{One-trees}
Relaxations play an important role in integer programming in general
and the TSP in particular. They provide lower bounds which can
be used to obtain quality guarantees for solutions. The prevalent
relaxation of combinatorial problems formulated as integer programs is
given by their LP relaxations. In several cases however, it is
possible to derive purely combinatorial relaxations.  In the case of
the \emph{symmetric} traveling salesman problem, a popular
combinatorial relaxation is given by \emph{one-trees}. A one-tree with
respect to a graph $G = (V, E)$ is given by a spanning tree of $G$
together with an edge adjacent to a distinguished source $1 \in V$.
Since every tour is a one-tree, the one-tree of minimum cost provides
a lower bound on the cost of a tour.

The computation of a one-tree in the static case involves the computation
of a minimum spanning tree (MST). This computation can be
performed
efficiently using Prim's algorithm. It is therefore natural to ask whether
this approach can be generalized to the time-dependent case.

Let $T=(V, F)$ be a spanning tree of the graph $G$. We direct the
edges in $T$ away from $1$. For each vertex $v \in V$, there exists a
unique $(1,v)$-path $P_v$ in $T$. Hence, there is a unique
arrival time $\theta^{\arr}(u)$ induced by $T$ for each $u \in V$. The total
cost of the edges in $T$ is then given by
\begin{equation}
  c(T) \define \sum_{(u, v) \in F} c_{u, v}(\arrtime(u)).
\end{equation}
A \emph{time-dependent minimum spanning tree (TDMST)} minimizes
$c(T)$. Unfortunately the computation of a TDMST is hard:
\begin{thm}
  \label{thm:tdmst_hardness}
  There is no $\alpha$-approximation algorithm for any $\alpha > 1$ for
  the TDMST problem unless $P = \NP$.
\end{thm}
\begin{proof}
  Consider an instance of the \textsc{3Sat} problem.
  Let $X \define \{x_1, \ldots, x_n\}$ be a set of $n$~variables and
  $\mathcal{Z} \define \{Z_1, \ldots, Z_m\}$ be a set of $m$ clauses, where
  each clause contains at most three literals.
  We construct a suitable instance of the TDMST problem
  using a number of components. First we define a component $A_i$ for each literal
  $x_i \in X$. The component is shown in Figure~\ref{pic:mst_proof_variable_component},
  the edges are annotated with their (static) travel times.
  Any spanning tree will arrive at $x_i$ either at time 1 or
  time $2$ depending on whether the resulting path leads past $s_i$ or not.
  \begin{figure}
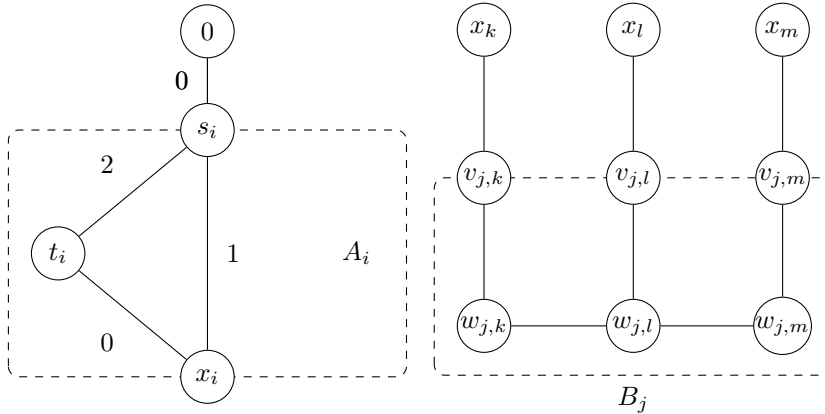

    \begin{center}
      \begin{subfigure}[b]{0.4\textwidth}
        \includegraphics[width=\textwidth]{TreeGadgetVariable.tikz}
        \caption{A component which queries whether a variable is set}
        \label{pic:mst_proof_variable_component}
      \end{subfigure}
      ~
      \begin{subfigure}[b]{0.4\textwidth}
        \includegraphics[width=\textwidth]{TreeGadgetClause.tikz}
        \caption{A component which determines whether a clause is satisfied}
        \label{pic:mst_proof_clause_component}
      \end{subfigure}
      \caption{Gadgets used in the proof of Theorem~\ref{thm:tdmst_hardness}}
    \end{center}
  \end{figure}
  Next we define a component $B_j$ for each clause $Z_j \in \mathcal{Z}$.
  Let $x_k$, $x_l$, $x_m$ be the literals which appear in $Z_j$. The edges
  in the component have the following travel times:
  \begin{enumerate}
  \item
    The edges between the vertices $w_{j, k}$, $w_{j, l}$ and $w_{j, m}$ have
    a constant travel time of 1.
  \item
    The edge connecting $v_{j, k}$ and $w_{j, k}$ has a travel time of $1$
    for times at most two, and $M \geq 1$ otherwise. The same holds
    true for the two other respective edges.
  \item
    The travel time of the edge connecting $x_k$ and $v_{j, k}$ depends
    on whether $x_k$ or $\overline{x}_k$ appears in the clause $Z_j$.
    In the former case the travel time is always 1, whereas in the latter
    it is given by
    \begin{equation}
      c_{x_k v_{j, k}}(\theta) \define
      \begin{cases}
        0, & \text{ if } \theta \geq 2 \\
        2, & \text{ otherwise}\\
      \end{cases}
    \end{equation}
  \end{enumerate}
  The instance including the components is depicted in
  Figure~\ref{pic:mst_proof_complete}.  Consider a satisfying truth
  assignment. For every literal $x_i$ set to \emph{true} we choose the
  path $0$, $s_i$, $x_i$, $t_i$ in component $A_i$. If the literal is
  set to \emph{false} we choose the path $0$, $s_i$, $t_i$, $x_i$.
  Thus, the arrival time at $x_i$ is $1$ if $x_i$ is set to
  \emph{true} and $2$ otherwise.  For each clause $Z_j$ we add the
  edges between the vertices corresponding to its literals and their
  respective $v_j$ counterparts. Since the clause is satisfied, the
  arrival time at at least one $v_j$ is $2$. Thus, the remaining part
  of $B_j$ can be spanned using three additional edges of cost $1$
  each.  The resulting tree hast costs of at most $2n + 6m$.
  Conversely, consider a tree with costs less than $M$. We first make
  the observation that any $x_i$ is connected by a path leading past
  $A_i$ to $0$. Otherwise, the path from $0$ to $x_i$ would lead past
  the component $B_j$ corresponding to some clause $Z_j$ containing
  $x_i$ or $\overline{x}_i$.  In this case however, it would not be
  possible to reach vertex $w_{j, i}$ before time $2$ and the cost of
  the tree would increase beyond $M$. Thus, any such tree corresponds
  to an assignment of variables. Every component $B_j$ is connected by
  an edge with cost of $1$. Therefore the assignment is also
  satisfying.  Assume there was an $\alpha$-approximation for the
  TDMST problem. We let $M \define \alpha (2n+6m) + 1$ and run the
  approximation. If the resulting tree has costs less than $M$, the
  \textsc{3Sat} instance is satisfiable.  Otherwise, the optimal TDMST
  has costs at least $M / \alpha > 2n + 6m$, \ie the instance is not
  satisfiable.

  \begin{figure}
    \centering
    \includegraphics[width=0.6\textwidth]{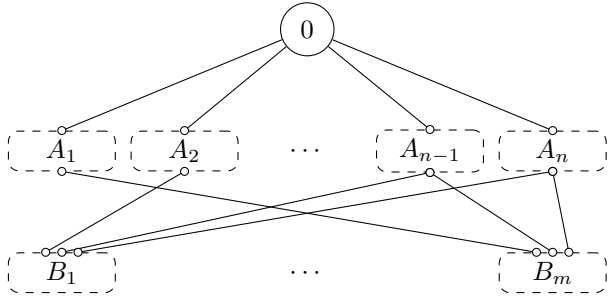}
    \caption{The TDMST construction used to prove Theorem~\ref{thm:tdmst_hardness}}
    \label{pic:mst_proof_complete}
  \end{figure}

\end{proof}

\section{Formulations}
\label{section:formulations}

\subsection{Time-expanded graphs}

In the following, we will consider formulations based on time-expanded graphs.
In order to introcude time-expanded graphs we first define a set of
reachable points in time. We let $\T : V \to 2^{\Theta}$,
\begin{equation}
  \begin{aligned}
    \T(v) \define \{ \theta \in \Theta \mid \: &
    \exists (a_1, \ldots, a_k), a_1 = (s,v_1), a_k = (u_k, v), \\
    & \arrtime(a_1, \ldots, a_k) = \theta \}
  \end{aligned}
\end{equation}
The time-expanded graph $D^{\T} = (V^{\T}, A^{\T})$ has vertices
$V^{\T} \define \{ v_{\theta} \mid v \in V,\, \theta \in \T(v) \}$
and arcs
\begin{equation}
  A^{\T} \define \{ (u_{\theta}, v_{\theta'}) \mid
  u_{\theta}, v_{\theta'} \in V^{\T},
  \theta' = \theta + c_{uv}(\theta) \}.
\end{equation}
We will denote an arc $(u_{\theta}, v_{\theta'})$ by $(u,v, \theta)$.
We will from now on assume that $c_{uv}(\theta) > 0$ for all
$(u, v) \in A$, $\theta \in \T(v)$. This directly implies
that $D^{\T}$ is acyclic.

\begin{exmp}
  Figure~\ref{pic:example} shows a directed graph with travel times for each arc
  and its time expansion. Any tour on $D$ can be embedded into $D^{\T}$
  as a $(s_0,s_{\theta})$-path.
  \begin{figure}[ht]
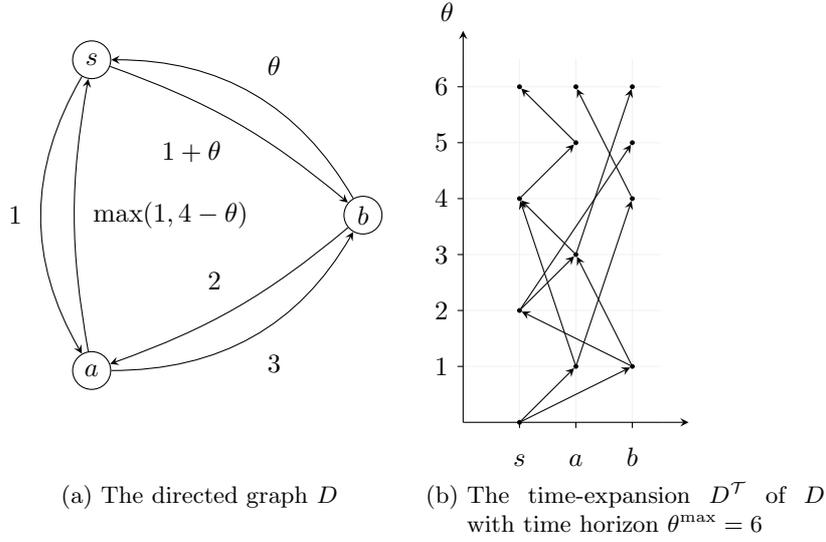

    \begin{center}
      \begin{subfigure}[t]{0.4\textwidth}
        \includegraphics[width=\textwidth]{ExampleGraph.tikz}
        \caption{The directed graph $D$}
        \label{subfig:example_graph}
      \end{subfigure}
      ~
      \begin{subfigure}[t]{0.4\textwidth}
        \includegraphics[width=0.66\textwidth]{ExpandedExample.tikz}
        \caption{The time-expansion $D^{\T}$ of $D$ with time horizon $\theta^{\max} = 6$}
        \label{subfig:expanded_example}
      \end{subfigure}
      \caption{A directed graph $D$ and its time-expansion $D^{\T}$.}
      \label{pic:example}
    \end{center}
  \end{figure}
\end{exmp}

\subsection{An Arc-based formulation}

We consider an arc-based formulation based on the graph $D^{\T}$
consisting of binary variables $x_{uv,\theta}$ for each arc in
$D^{\T}$.  The resulting formulation is inspired by
a three-index-formulation for STDTSP~\cite{Picard1978}. The formulation
consists of a flow through the time-expanded graph $D^{\T}$, which has
to cover each vertex exactly once.
\begin{equation}
  \label{eq:arc_based}
  \begin{aligned}
    \min\ & \sum_{(u,v,\theta) \in A^{\T}} c_{uv,\theta} \cdot x_{uv, \theta} && \\
    &\sum_{\theta \in \T(v)} \sum_{(v, w, \theta) \in \delta^{+}(v_{\theta})} x_{vw, \theta} =1 &&\text{for all } v \in V \\
    & \sum_{(v,w, \theta) \in \delta^{+}(v_\theta)} x_{uv,\theta} - \sum_{(u, v, \theta') \in \delta^{-}(v_{\theta})} x_{uv,\theta'}
    = 0
    && \text{for all } v \ne s, \theta \in \T(v) \\
    &x_{vw,\theta} \in \set{0,1} && \text{for all } v \neq w, \theta \in \T(v).
  \end{aligned}
\end{equation}

\begin{rem}
  \label{rem:flow}
  Any solution of the IP or its LP-relaxation can be decomposed into
  a set of paths leading from vertex $s_0$ to $s_{\theta}$ for
  $\theta > 0$. Thus, an equivalent cost function is given by
  $\sum_{\theta \in \T(s)} \sum_{(v,s,\theta') \in \delta^{-}(s_\theta)} \theta \cdot  x_{vs,\theta}$.
\end{rem}

\paragraph{Relation to the static ATSP}

In the following, we will consider the relationship between the \tdtsp and the
static ATSP problem. To this end, we let $x : A \to \R_{\geq 0}$ be the combined
flow traversing an arc $(u, v) \in A$, i.e.
\begin{equation}
  x_{uv} \define \sum_{\theta \in \T(u)} x_{uv, \theta}
\end{equation}
where $(x_{uv,\theta})_{(u, v, \theta) \in A^{\T}}$ is a feasible
solution of \eqref{eq:arc_based}. Observe that the covering
constraints and the flow conservation yield the well-known 2-matching
equations $x(\delta^{+}(v)) = x(\delta^{-}(v)) = 1$ for all $v \in V$.
Similarly, integrality together with the condition $x_{uv} \leq 1$
follows from the integrality of the original solution. However, a
correct static ATSP formulation still requires subtour
elimination constraints (SECs) of the form
\begin{equation}
  x(\delta^{+}(S)) \geq 1 \quad \forall S \subset V,\ S \neq \emptyset, V.
\end{equation}
Since $D^{\T}$ is acyclic, any solution of \eqref{eq:arc_based} is
guaranteed to satisfy the additional SECs\footnote{Equivalently,
flow augmentation techniques such as \cite{Gouveia1995} used to strengthen
ATSP formulations are redundant for solutions of the \tdtsp.}.
Still, SECs are not necessarily satisfied by fractional
solutions. Thus, formulation \eqref{eq:arc_based} can be strengthened
by separating SECs with respect to the underlying static ATSP.

We can produce fractional solutions to the static ATSP problem by
computing the combined flow after having successfully separated all SECs.
Consequently, we can use any ATSP separator to derive valid inequalities for the
ATSP which we can then formulate in terms of the variables corresponding
to $D^{\T}$ in order to strengthen our formulation.

Note that while any feasible \tdtsp solution is feasible for the underlying
ATSP, generic ATSP solutions do not necessarily produce feasible solutions
of the \tdtsp.
Specifically, no tour $T=(a_1,
\ldots, a_n)$ with $\arrtime(T) > \theta^{\max}$ can be embedded into
$D^{\T}$. The complete description of the \tdtsp in terms of combined
variables can be obtained by adding \emph{forbidden path} constraints
of the form
\begin{equation}
  \sum_{a \in P} x_a \leq k - 1 \quad
  \forall \: P = (a_1, \ldots, a_k) : \arrtime(P) > \theta^{\max}.
\end{equation}
As a result, facet-defining ATSP inequalities, while valid, are
not necessarily facet-defining for the \tdtsp.

\begin{lem}[Dimensionality]
  Let $n^{\T} \define |V^{\T}|$, $m^{\T} \define |A^{\T}|$ be the
  number of vertices and arcs, respectively, of $D^{\T}$, and $V^{s}
  \define \{ s_{\theta} \in V^{\T} \}$ the $n^{s} \define |V^{s}|$
  many vertices corresponding to $s$. If $c$ satisfies the
  time-dependent triangle inequality
  \eqref{eq:time_dependent_triangle}, then
  \begin{equation}
    \dim(P) \leq m^{\T} - (n^{\T} - n^{s}) - (n - 1)
  \end{equation}

\end{lem}
\begin{proof}
  The dimension of $P$ is trivially bounded by the difference between the
  number of variables $m^{\T}$ and the rank of the system of
  equations \eqref{eq:arc_based}. Since $D^{\T}$ is acyclic, the
  system of equations ensuring flow conservation on the
  vertices in $V^{\T} \setminus V^{s}$ has full rank of $n^{\T} - n^{s}$.
  Consider a vertex $v \neq s$ contained in a tour
  $T=(v_1 = s, \ldots, v_{i - 1}, v_{i} = v, v_{i + 1}, \ldots, v_n)$
  given as a sequence of vertices in $D$. Assume that
  equation $x(\delta^{+}(v)) = 1$ is struck from system \eqref{eq:arc_based}.
  Since $c$ satisfies the time-dependent triangle inequality,
  the sequence $T'=(v_1 = s, \ldots, v_{i - 1}, v_{i + 1}, \ldots, v_n)$
  is a feasible solution to the reduced system of equations, yet
  it is not a feasible \tdtsp solution. Therefore, none of
  the $n - 1$ equations corresponding to vertices other than
  $s $ can be struck without increasing
  dimensionality. Thus, the combined system of equations has
  the required rank.
\end{proof}

\subsection{Pricing}

The approach of time-expansion can be used to solve a variety of
time-dependent problems~\cite{QuickestFlows,Railroad}.
Unfortunately, the  time-expansion
of a problem quickly increases the size of the resulting formulations.
Specifically, in case of the \tdtsp, even moderately sized instances
of less than one hundred vertices can result in millions of arcs in $A^{\T}$,
making it difficult to solve even the LP-relaxation of the \tdtsp.

To alleviate the problem, an obvious approach is to use
column-generation (see \cite{CGSummary} for a summary on the topic).
Nevertheless, there are a number of different variants of CG, especially
regarding the pricing strategy. It is not immediately clear which one is the
best for the \tdtsp. We will present the tested approaches in the following.

Let $(\lambda_v)_{v \in V}$, $(\mu_{v_{\theta}})_{v \neq s,
\theta \in \T(v)}$ be
the dual variables of the respective constraints in~\eqref{eq:arc_based}.
The reduced cost of an arc $(v, w, \theta)$ is then given by
\begin{equation}
  \overline{c}_{vw,\theta} \define c_{vw}(\theta) -
  \left( \lambda_{v} + \mu_{v_{\theta}} - \mu_{w_{\theta + c_{vw}(\theta)}} \right).
\end{equation}
To obtain a feasible solution to populate the initial LP, we compute
a heuristic tour, which we then add as a whole.

\paragraph{Lagrangean pricing} While the pricing approach
significantly reduces the formulation size and facilitates the
solution of much larger instances, the approach can be significantly
improved. Consider a single arc $(v, w, \theta)$ with negative reduced
costs: The arc can only obtain a positive value in the subsequent
LP-solution if it is part of a $(s_0,s_{\theta})$-path.  It is therefore
advisable to generate entire paths at once rather than single arcs.
The pricing problem then becomes a shortest path problem in $D^{\T}$.
Even though the reduced costs are negative, the pricing problem can
be solved using breadth-first search since $D^{\T}$ is acyclic.
We also employ a technique known as
\emph{Lagrangean pricing}~\cite{LagrangeanPricing}.
The technique is based on the observation that the pricing problem
is a Lagrange relaxation of the full LP, which implies that
the difference between the current LP value and the value of the
full LP is bounded by the minimum reduced cost of, in this case,
an $(s_0,s_{\theta})$-path. The pricing loop is aborted as
soon as the cost rises above a value of $-\epsilon$. This approach
helps to deal with the degeneracy often present in formulations
of combinatorial optimization problems by avoiding to price
variables which have negative reduced costs without attaining
a nonzero value in the optimal basis of the LP relaxation.
It is also the case that paths obtained from the pricing procedure
occasionally correspond to tours in $D$ and therefore
to feasible \tdtsp solutions.

\paragraph{Pricing cycle-free paths}
\label{paragraph:cycle_free}
Unfortunately, many paths which are generated throughout the pricing
do not share much resemblance with tours in the underlying graph
$D$: On the one hand certain paths only contain few vertices and
lead almost immediately back to $s_{\theta}$. We will address this
problem using the propagation of lower bounds.  On the other hand,
paths frequently contain cycles with respect to $D$, \ie they
contain two different versions $v_{\theta}$, $v_{\theta'}$ of the
same vertex $v \neq s$. It is of course possible to generate
inequalities in order to cut off a fractional solution $\tilde{x}$
containing a cycle in its path decomposition (see
Subsection~\ref{subsection:inequalities}). However, ideally we
would like not to have paths containing cycles in the LP in the
first place. Obviously, the problem of finding an acyclic path of
negative reduced cost is equivalent to finding the optimal solution
to the \tdtsp problem. It is however possible to find $k$-cycle free
paths, \ie paths not containing a cycle with at most $k$ arcs using
Dijkstra-like labeling schemes~\cite{AcyclicShortestPath}.
Specifically, avoiding 2-cycles merely increases computation time by
a factor of two, while significantly improving the resulting lower
bounds.  It is also possible to avoid $k$-cycles for arbitrary $k$;
however, the proposed algorithm takes $\O((k!)^2)$ time, which
quickly makes the approach intractable for increasing values of $k$.

\subsection{Valid inequalities}
\label{subsection:inequalities}

In order to strengthen the formulation, a number of additional inequalities
can be included in the formulation. We give a brief summary of valid inequalities,
some of which are well-known ATSP inequalities, whereas others are either adaptations
of STDTSP inequalities or newly derived ones.

\paragraph{ATSP inequalities}
\label{paragraph:dkp}

Apart from the subtour elimination constraints, the probably best-known family of
facet-defining inequalities for the ATSP goes by the name of
$D_{k}^{+}$-inequalities \cite{LineareCharatkerisierung,TSPVariations}.
$D_{k}^{+}$-inequalities are defined on a complete directed graph $D=(V, A)$ with $n$ vertices.
To simplify notation, for sets $S,T \subseteq V$ we let
$[S:T] \define \{ (u, v) \in A \mid u \in S, v \in T\}$.
The $D_{k}^{+}$-inequality for a sequence $(v_1, \ldots, v_k)$ of $2 \leq k < n$
distinct vertices is given by
\begin{equation}
  \begin{split}
    \sum_{j=1}^{k -1} x_{v_{j},v_{j+1}} + x_{v_k,v_1}
    + 2 x([\{v_1\} : \{v_3, \ldots, v_k\}]) & \\
    + \sum_{j = 4}^{k} x([\{v_j\} : \{v_3, \ldots, v_{j - 1}\}]) \leq k - 1.
  \end{split}
\end{equation}
The separation of $D_{k}^{+}$-inequalities involves the enumeration of
possible sequences in a branch and bound-like fashion. Nonetheless,
the separation works well in practice, since many of the possible
sequences can be pruned. Note that in the special case $k = 2$
the inequality becomes $x_{uv} + x_{vu} \leq 1$.

\paragraph{Incompatibilites}
\label{paragraph:odd_cat}

Since any feasible solution to an integer program is a stable set with
respect to its incompatibility graph $\I$, cliques and odd cycles in $\I$ are
the basis for many strong inequalities for arbitrary integer programs, a
fact which is often used in MIP solvers.
While the incompatibility graph of the symmetric TSP problem is empty, the
ATSP problem already has a significant amount of incompatibilities. Specifically,
the arcs $(u, v) \neq (u',v')$ are incompatible if $v = v'$ or $u = u'$ or
both $u = v'$ and $u' = v$. Clique inequalities are implied by the constraints
$x(\delta^{+}(v)) = x(\delta^{-}(v)) = 1$ and $x_{uv} + x_{vu} \leq 1$.
However, it is possible to derive inequalities from odd cycles in $\I$.
These \emph{odd closed alternating trails}~\cite{NewFacets} (odd
CATS for short) can be separated heuristically by computing shortest
paths in an auxiliary bipartite graph. Note that with respect to the
incompatibility graph of the \tdtsp, the cuts correspond to
odd cycles of cliques rather than odd cycles of vertices. As a result the
obtained cuts are stronger than ordinary odd cycle cuts and easier
to separate due to the small size of the incompatibility graph of the
underlying ATSP.

\paragraph{Odd path-free inequalities}
\label{paragraph:odd_path_free}
Consider a set $S \subseteq V \setminus \{s\}$ of vertices
of the original graph. Let
$V^{\T}(S) \define \{u_{\theta} \in V^{\T} \mid u \in S \}$
be the corresponding vertices in $D^{\T}$ and $A^{\T}$
the induced subgraph:
\begin{equation}
  A^{\T}(S) \define \{ (u, v, \theta) \in A^{\T} \mid u, v \in S \}.
\end{equation}
The intersection of any tour with the set $A^{\T}$ can contain at most $|S| - 1$ arcs,
the corresponding inequality is equivalent to a subtour elimination constraint.
If $|S| = 2k + 1$ is odd, another inequality
can be obtained by considering certain subsets of $A^{\T}$. Specifically,
$A' \subseteq A^{\T}$ is called \emph{path-free}, if it does not
contain a path consisting of at least three different vertices in $S$. The intersection
of a path-free set $A'$ with any tour can contain at most $k$ arcs,
yielding the following \emph{odd path-free} inequality (see Figure~\ref{pic:path_free}):
\begin{equation}
  \sum_{(u, v, \theta) \in A'} x_{uv, \theta} \leq k.
\end{equation}
In order to separate an odd path-free inequality we first have to
find some \emph{promising} subset $S$, i.e. a set off odd size forming
a clique of sufficient weight. For such a set $S$ the separation problem
is equivalent to finding a stable set of maximum weight in
the (undirected) line graph of $(V^{\T}(S), A^{\T}(S))$.

Since both problems are $NP$-hard themselves, we make several
restrictions in order to decrease the computational costs.
First, note that larger values of $k$ result in larger
line graphs making the computation of stable sets much
more challenging. We therefore chose to restrict ourselves
to the case of $k = 1$ (in which odd path-free sets
correspond to cliques in the incompatibility graph $\I$).
This restriction also enables us to find
promising sets in polynomial time by enumerating all
3-sets of vertices in $D$. In order to avoid separating
very similar inequalities we consider only
the largest promising 3-set containing
each vertex $v \in D \setminus \{s\}$. The separation
for each set is performed using an integer program.
\begin{figure}[ht]
  \centering
  \includegraphics[width=0.25\textwidth]{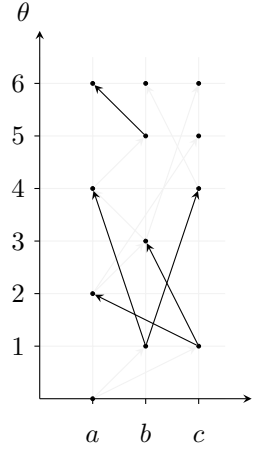}
  \caption{A path-free set of arcs on three vertices}
  \label{pic:path_free}
\end{figure}

\paragraph{Lifted subtour elimination inequalities}
\label{paragraph:lsec}

As discussed above, subtour elimination constraints
can be used to cut off fractional \tdtsp solutions.
SECs can be separated in polynomial time by solving a series
of flow problems on the underlying graph, ultimately yielding
a set $S \subseteq V$, $S \neq \emptyset, V$ maximizing
$x(\delta^{+}(S))$. In the following we will assume w.l.o.g that
$s \in S$. SECs can be strengthened by imposing an upper limit on
the time of the initial departure from the set $S$.
After all, any tour has to leave $S$ sufficiently early
to be able to reach the vertices in $V \setminus S$ and return to
$s$. More formally, let $\hat\theta$ be such that
\begin{equation}
  \hat\theta \geq
  \max \{ \theta \mid \text{There exists a tour $T$ leaving $S$ for the first time at $\theta$} \}.
\end{equation}
Then, the following \emph{lifted} subtour elimination constraint
(LSEC) is valid for all tours:
\begin{equation}
  \sum_{\substack{ (u, v, \theta) \in A^{\T} \\ u \in S, v \notin S, \\ \theta \leq \hat\theta }} x_{uv, \theta} \geq 1
\end{equation}
Clearly, $\hat\theta$ is maximized for a tour which first serves $S$,
then $V \setminus S$ and returns to $s$ immediately afterwards. Thus,
maximizing $\hat\theta$ would involve the
solution of a series of \tdtsp problems on $V \setminus S$,
an approach which is clearly intractable in practice. We propose to
compute a larger value of $\hat\theta$ given by $\theta^{\max} - \check\theta$
with a lower bound $\check\theta$ on the length of the shortest tour on $V \setminus S$.
We derive $\check\theta$ by considering an ATSP on $V \setminus S$ with costs given
by suitably chosen lower bounds on the travel times $c_{uv, \theta}$. The ATSP itself can
be bounded from below by computing an arborescence of minimum weight.

\paragraph{Cycle inequalities}
\label{paragraph:cycle}
While the formulation (\ref{eq:arc_based}) does not contain any subtours it is
possible that a path in the LP-relaxation visits a vertex at several
different points in time. Specifically, a path $P$ in $D^{\T}$ can be
of the form $P = (\ldots, (u, v, \theta), (v, w, \theta'), \ldots)$,
forming a cycle of length 2 in
$D$ (where $v \neq s, \theta' = \theta + c_{uv}(\theta)$).
We know that if a tour visits $v$ at $\theta'$,
it has to go on using an arc $(v, w, \theta')$ such
that $w \neq u$. Hence the following inequality is valid:
\begin{equation}
  x_{uv,\theta} \leq \sum_{w \neq u,v} x_{vw,\theta'}
\end{equation}
More generally, consider a path which contains the sequence
$(u_1, v_1, \theta_1), \ldots, (u_k, v_k, \theta_k)$ such that
$v_1 = v_k$. In this case it makes sense to add the following
inequality:
\begin{equation}
  x_{u_1v_1,\theta_1} \leq \sum_{j=1}^{k-1} \sum_{v \notin \{u_1, v_1, \ldots, v_j \}} x_{v_j v,\theta_{j+1}}
\end{equation}
In order to separate these $r$-cycle inequalities it is convenient
to consider a path-decomposition of the flow through the network
and to eliminate $r$-cycles from the individual paths.

\paragraph{Unitary AFCs}
\label{paragraph:unitary_afc}
Unitary AFCs (admissible flow constraint) were introduced in
\cite{TimeDependentTSPPoly} for the STDTSP. In the context of the
\tdtsp they can be explained as follows: Consider an arc $(u, v,
\theta)$ with $u, v \neq s$ carrying a nonzero flow. The flow enters
some set of vertices which has to be left again in order to reach the
source $s$. Specifically, let $X \subseteq V^{\T}$ such that
\begin{enumerate}
\item
  $X$ contains $v_{\theta + c_{uv}(\theta)}$.
\item
  Every vertex $(v', \theta') \in X$ is reachable
  from $(v, \theta + c_{uv}(\theta))$
  using only arcs in the graph induced by $X$.
\item
  The set $X$ contains no copies of the vertices $u, v, s$.
\end{enumerate}
In this case we can add the following inequality:
\begin{equation}
  x_{uv, \theta} \leq \sum_{\substack{(u', v', \theta') \in \delta^{+}(X) \\ v' \neq u, v}} x_{u'v',\theta'}
\end{equation}
In order to separate these types of inequalities we consider for
a fixed arc $(u, v, \theta)$ all vertices which are
reachable from $(v, \theta + c_{uv}(\theta))$. We then solve a series
of min-cut problems with capacities according to the fractional solution.
If we find a cut with a value of less than one we add it to the LP.

\subsection{Speedup techniques}

The addition of cutting planes already significantly strengthens
formulation \eqref{eq:arc_based}. There are however several other
techniques which can be used to speed up the computation of the
optimal tour in a branch-cut-and-price framework:

\paragraph{Propagation}
\label{paragraph:propagation}
At any given step in the solution process we have a (local) dual bound
$\underline{\theta}$ given by the value of the
LP-relaxation (of the current node in the branch-and-bound tree) and
a primal bound $\overline{\theta}$ given by the currently best known
integral solution. Clearly, any arc $(v, w, \theta)$ with
$\theta > \overline{\theta}$ can be fixed to
zero, as can any arc $(v, s, \theta)$ with
\begin{equation}
  \theta + c_{vs}(\theta) < \underline{\theta}.
\end{equation}
As these bounds become more accurate, more and more arcs can
be discarded. The relaxation can often be strengthened significantly
by employing this technique since the LP-relaxations frequently
consists of paths which send an amount of flow for $s_0$ back
to $s_{\theta}$ via a path with containing very few vertices and
leading back into $s$ at a time lower than $\underline{\theta}$.

\paragraph{Compound branching}
Traditionally, a branch-and-bound approach would branch on individual
variables $x_{uv, \theta}$, leading to a highly unbalanced
branch-and-bound tree. We instead propose to branch on the combined
flow $(x_{uv})_{(u,v) \in A}$. We incorporate the incompatibilities with respect to
the underlying ATSP in order to increase the dual bounds in child
nodes.  Specifically, whenever an arc $(u, w)$ is fixed to one during
the branching, we fix every incompatible arc $(u', v')$ to zero. We
incorporate the branching rule into the pricing loop by ignoring arcs
fixed to zero and incorporating the dual costs of arcs which have been
fixed to one.

\paragraph{Primal heuristics}
\label{paragraph:primal_heuristics}

In order to obtain improved primal solutions we a simple
heuristic based on the current LP-solution
$(x^{*}_{uv, \theta})_{(u,v, \theta) \in A^{\T}}$.
Specifically, we construct a path $P$ traversing
$D^{\T}$ starting at $s_0$ by
appending arcs to vertices whose counterparts in $D$ are still
unexplored by $P$ until the path forms a tour in $D$.  During the
construction of the tour we disregard arcs fixed to zero by the
compound branching rule introduced above. If there are multiple arcs
to choose from, we compute scores using the following metrics:
\begin{enumerate}
\item
  We score arc $(u, v, \theta)$ by the inverse of
  its travel time $c_{uv}(\theta)$.
\item
  We evaluate $(u,v, \theta)$ according to the value of $x^{*}_{uv, \theta}$
  using travel times to break ties.
\item
  We measure $(u, v, \theta)$ using the combined value $x^{*}_{uv}$ using
  a similar tie-breaking rule.
\end{enumerate}
Note that the iterative construction of paths in $D^{\T}$ is computationally
inexpensive. Thus, to increase the chance of finding an improved tour,
we randomize the selection of arcs based on probabilities proportional
to the different score functions and perform several runs using
different random seeds.

\subsection{A path-based formulation}

Recall that any feasible solution of the \tdtsp problem
\eqref{eq:arc_based} can be decomposed into paths from $s_0$ to
$s_{\theta}$ for some $\theta \in \Theta$. With respect to $D$ these
paths correspond to cycles containing vertex $s$. We let $\mathcal{P}$ be
the set of paths, let $\alpha_{v,P} \define |\{ \theta \mid (v, w, \theta) \in P \}|$
and reformulate the problem in terms of individual paths:
\begin{equation}
  \label{eq:path_based}
  \begin{aligned}
    \min\ &\sum_{P \in \mathcal{P}} c_P x_P \\
    &\sum_{P \in \mathcal{P}} \alpha_{v,P} \cdot x_P = 1
    &&\text{for all } v \in V \\
    &x_P \in \{0, 1\} && \text{for all } P \in \mathcal{P} \\
  \end{aligned}
\end{equation}
Note that any solution of this IP consists of a single variable
$x_P$ set to $1$ and all others set to $0$, in which case $P$ must
correspond to a tour. Any fractional solution consists of at most
$n$ different paths which need not be tours in $D$.
The resulting system is small in terms of the number
of constraints at the expense of the number of variables. Thus,
a pricing approach is absolutely necessary in this case. Since arc-based
and path-based solutions are equivalent, all previously discussed techniques
can be easily adapted to the path-based formulation.

\section{Computational experiments}
\label{section:computational}

\subsection{Instances}

In order to test different formulations and techniques we generated
several problem instances, each given by a directed complete graph and
cost functions associated with its arcs. We embedded the vertices of
the graph into $\{0, \ldots ,100\}^{2}$ and introduced (symmetric)
costs $c_a$ using rounded-down euclidean distances between the points
of the embedding.

We then augmented the static costs to time-dependent functions $c_a(\theta)$.
We first added $M \in N$ time steps $\theta_{1} < \theta_{2} < \ldots < \theta_{M}$
within the range $\{0, \ldots, \theta^{\max}\}$ and used these time
steps to construct a piecewise linear function
$f_a(\theta) : \mathbb{N} \to \mathbb{Z}$:
\begin{enumerate}
\item
  We let $f_a(0) \define 0$ and fixed the slope at zero to $+1$.
\item
  The slope alternates between $+1$ and $-1$ with break points
  at $\theta_i$ for $i = 1,\ldots, M$.
\end{enumerate}
We let $\lambda > 1$ and define the cost function
$c_a : \{0, \ldots, \theta^{\max}\} \to \N$ as
\begin{equation}
  c_a(\theta) \define c_a + \max(\min(f_a(\theta), \lambda c_a), 0)
\end{equation}
The parameter $\lambda$ controls the multiple of $c_a(\theta) / c_a$ which
can be attained. We generally let
$\lambda = 3$ (see Figure~\ref{pic:func} for an example) and distribute
$M = 100$ break point over an interval of 1000 points in time.
\begin{center}
  \begin{figure}
    \begin{center}
      \begin{tikzpicture}
        \begin{axis}[xmin=0,ymin=0, ymax=40,xmax=100,grid=major, grid style={thin,black!5}, unit vector ratio=1 1]
          \addplot[mark=none, black] table [x=x, y=y, col sep=comma] {Data/func.plot};
        \end{axis}
      \end{tikzpicture}
      \caption{A sample plot of the travel time function for costs of 10, a time
        horizon of $\theta^{\max} = 100$ and $M = 10$ break points.
        The cost is constrained by a factor of $\lambda = 3$}
      \label{pic:func}
    \end{center}
  \end{figure}
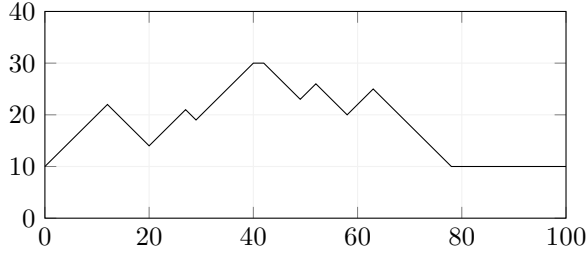
\end{center}
Note that the functions defined above satisfy the FIFO property. By using
shortest-path distances we ensure that the time-dependent triangle inequality
is satisfied as well.

\subsection{Formulations}

We implemented the different formulations based on the
SCIP\cite{SCIP} IP solver\footnote{SCIP version 4.0.0 with SoPlex 3.0.0 as an
  LP-solver}. We ran all experiments on an Intel Core i7 CPU clocked
at \SI{3.2}{\giga\hertz} on a system with \SI{8}{\giga\byte} of
RAM. We started with 50 relatively small instances containing 20 vertices
each. We first computed optimal tours with respect to the static
costs and used the resulting travel times with respect to the
time-dependent cost to derive smaller time horizons to decrease the
size of the corresponding time-expanded graphs. Despite their
relatively small size, the time-expansions were quite large, each
containing between \num{80000} and \num{170000} arcs.

We started by comparing the different combinations of formulations
and pricing approaches. We were for the most part not able to
determine the optimal solutions within the prescribed time
limit of \SI{3600}{\second}. There are however significant
differences between the different
formulations (see Table~\ref{table:formulations} for details).
\begin{itemize}
\item
  The different pricing approaches differ significantly performance-wise.
  Specifically, the arc-based pricing approach fails to solve even a single
  root LP.
\item
  The path-based formulation \eqref{eq:path_based} generally yields
  smaller LPs than the full arc-based formulation. However, similarly
  to the arc-based pricing approach, many root LPs are not solved within
  the time limit.
\item
  The path-based pricing approaches reduce the average size of the solved LPs
  by about 90 \%. As a result, almost all of the root LPs are solved successfully.
\item
  Using a dual stabilization approach does not result in smaller gaps
  compared to the simple path-based pricing.
\item
  The most successful approach is based on pricing 2-cycle free paths.
  While the average LP size does not change much, the remaining gap
  after the exhaustion of the computation time is decreased furthest
  when pricing 2-cycle free paths. This is due to the fact that,
  as mentioned above, the improved dual bound more than makes
  up for the increased computational time required to
  avoid 2-cycles during the pricing of new paths.
\end{itemize}

\begin{center}
  \newcommand{\DataRow}[1]{
  \csvreader[head to column names,separator=semicolon]{
    Data/Formulations/#1.csv}{}{
    \numSolved &
    \numLPSolved &
    \remainingGap &
    \numCols &
    \numRows}}

\newcommand{\mcrot}[4]{\multicolumn{#1}{#2}{\rlap{\rotatebox{#3}{#4}~}}}
\newcommand*\rot{\rotatebox{90}}

\begin{table}[htb]

  \caption{Solving statistics for different formulations, based on 50 different instances consisting of 20 vertices each.}
  \label{table:formulations}

  \renewcommand{\arraystretch}{1.25}
  \setlength\tabcolsep{4pt}
  \begin{tabular*}{\textwidth}{lrrrrrr}

    \fatline

    & \mcrot{1}{c}{45}{Number of instances solved} & \mcrot{1}{c}{45}{Number of root LPs solved}  & \mcrot{1}{c}{45}{Remaining gap} & \mcrot{1}{c}{45}{Number of columns}  & \mcrot{1}{c}{45}{Number of rows} \\

    \noalign{\vskip 1mm}
    \hline
    \noalign{\vskip 1mm}

    Arc-based & \DataRow{simple_program} \\
    Arc-based  pricing & \DataRow{sparse_edge_pricer} \\
    Path-based pricing & \DataRow{sparse_program} \\
    Stabilizing pricing & \DataRow{sparse_stabilizing_pricer} \\
    \textbf{2-cycle free pricing} & \DataRow{sparse_simple_acyclic_pricer} \\
    Path-based formulation & \DataRow{path_based_program} \\
    \fatline

  \end{tabular*}
\end{table}

\end{center}

\subsection{Valid inequalities and primal heuristics}

We proceed to study the effect of adding valid inequalities in order
to increase dual bounds. To this end we restrict ourselves to the
arc-based formulation with 2-cycle free path pricing, which
performed best in the experiments conducted this far. In order
to evaluate the effectiveness of different classes of valid
inequalities we again consider the remaining gap after \SI{3600}{\second}
of computations. The remaining gap is a good measure of the overall
effectiveness of the different classes of inequalities, since it
strikes a balance between the increase of the dual bound and
the required separation time. The latter can be substantial, in
particular if the separation involves the solution of $\NP$-hard problems.
We make the following observations (see the details in
Table~\ref{table:inequalities}):
\begin{itemize}
\item
  The separation of cycle inequalities actually increases the gap compared
  to the formulation without any separation. It is therefore inefficient
  to consider these inequalities at all.
\item
  There is no significant decrease with respect to the remaining gap
  when separating unitary AFC, odd path-free, and odd CAT inequalities.
  Apparently, the separation time for these classes of inequalities does
  not merit the increased dual bounds.
\item
  By far the most efficient classes of inequalities are (lifted) subtour
  elimination constraints. Few inequalities suffice to significantly
  decrease the remaining gap.
\item
  Adding primal heuristics decreases the remaining gap by a considerable margin.
  It is particularly efficient to construct tours based on the combined
  flow $x_{uv}^{*}$ of the current LP relaxation. Apparently the LP is able
  to accurately determine the underlying arcs which are contained in
  tours with small travel times. In contrast the built-in heuristics
  seem to be unable to take advantage of this fact.
\item
  The propagation of upper and lower bounds yields an additional improvement
  on the running times. The combination of the speedup techniques makes it
  possible to solve more than ten percent of the instances to optimality.
\end{itemize}

\begin{center}
  \newcommand{\DataRow}[1]{
  \csvreader[head to column names,separator=semicolon]{
    Data/Inequalities/#1.csv}{}{
    \numSolved &
    \remainingGap &
    \runningTime}}

\newcommand{\mcrot}[4]{\multicolumn{#1}{#2}{\rlap{\rotatebox{#3}{#4}~}}}
\newcommand*\rot{\rotatebox{90}}

\begin{table}[htb]

  \caption{Solving statistics for different combinations of speedup techniques, based on 50 different instances consisting of 20 vertices each.}
  \label{table:inequalities}

  \renewcommand{\arraystretch}{1.25}
  \setlength\tabcolsep{4pt}
  \begin{tabular*}{\textwidth}{lrrrrrr}

    \fatline

    & \mcrot{1}{c}{45}{Number of instances solved} & \mcrot{1}{c}{45}{Remaining gap} & \mcrot{1}{c}{45}{Running time} \\


    \noalign{\vskip 1mm}
    \hline
    \noalign{\vskip 1mm}

    \hyperref[paragraph:cycle]{Cycle} & \DataRow{cycle} \\
    \hyperref[paragraph:dkp]{$D_k^{+}$} & \DataRow{dk} \\
    \hyperref[paragraph:lsec]{LSEC} & \DataRow{lifted_subtour} \\
    \hyperref[paragraph:odd_cat]{Odd CAT} & \DataRow{odd_cat} \\
    \hyperref[paragraph:odd_path_free]{Odd path-free} & \DataRow{odd_path_free} \\
    SEC & \DataRow{subtour} \\
    \hyperref[paragraph:unitary_afc]{Unitary AFC} & \DataRow{unitary_afc} \\

    \noalign{\vskip 1mm}
    \hline
    \noalign{\vskip 1mm}

    \hyperref[paragraph:lsec]{LSEC} + \hyperref[paragraph:primal_heuristics]{Primal heuristics} & \DataRow{heuristics} \\
    \textbf{\hyperref[paragraph:lsec]{LSEC} + \hyperref[paragraph:primal_heuristics]{Primal heuristics} + \hyperref[paragraph:propagation]{Propagation}} & \DataRow{propagator} \\

    \fatline

  \end{tabular*}
\end{table}

\end{center}

\subsection{Combinations of inequalities}

It is clear from the previous experiments that the addition of SECs / LSECs is the most
effective approach to improve the improve dual bounds. Together with primal heuristics
and objective value propagation it is possible to reduce gaps to the point
of being able to solve a significant part of all instances. We go on to study
the effect of combining SECs / LSECs with other classes of inequalities. To this
end we first separate SECs / LSECs to strengthen the LP-relaxation before applying
separation procedures for different classes on inequalities while employing
both primal heuristics and objective value propagation. The results are depicted
in Table~\ref{table:combined}. Unfortunately, the effects of separating additional
inequalities from the strengthened relaxation has no significant effect on either
the amount of instances solved to optimality or the remaining gap for unsolved instances.

\begin{center}
  \newcommand{\DataRow}[1]{
  \csvreader[head to column names,separator=semicolon]{
    Data/Combined/#1.csv}{}{
    & \numSolved &
    \remainingGap &
    \runningTime}}

\newcommand{\mcrot}[4]{\multicolumn{#1}{#2}{\rlap{\rotatebox{#3}{#4}~}}}
\newcommand*\rot{\rotatebox{90}}

\begin{table}[htb]

  \caption{Solving statistics for different combinations of speedup techniques, based on 50 different instances consisting of 20 vertices each.}
  \label{table:combined}

  \renewcommand{\arraystretch}{1.25}
  \setlength\tabcolsep{4pt}
  \begin{tabular*}{\textwidth}{p{4cm}rrrrrr}

    \fatline

    & & & \mcrot{1}{c}{45}{Number of instances solved} & \mcrot{1}{c}{45}{Remaining gap} & \mcrot{1}{c}{45}{Running time} \\


    \noalign{\vskip 1mm}
    \hline
    \noalign{\vskip 1mm}

    \multirow{4}{*}{
      \begin{tabular}{l}
        \hyperref[paragraph:lsec]{LSEC} +\\
        \hyperref[paragraph:primal_heuristics]{Primal heuristics} +\\
        \hyperref[paragraph:propagation]{Propagation}
      \end{tabular}
      }

    & \hyperref[paragraph:dkp]{$D_k^{+}$} & \DataRow{lifted_subtour_dk} \\
    & \hyperref[paragraph:odd_cat]{Odd CAT} & \DataRow{lifted_subtour_odd_cat} \\
    & \hyperref[paragraph:odd_path_free]{Odd path-free} & \DataRow{lifted_subtour_odd_path_free} \\
    & \hyperref[paragraph:unitary_afc]{Unitary AFC} & \DataRow{lifted_subtour_unitary_afc} \\

    \noalign{\vskip 1mm}
    \hline
    \noalign{\vskip 1mm}

    \multirow{4}{*}{
      \begin{tabular}{l}
        \hyperref[paragraph:sec]{SEC} +\\
        \hyperref[paragraph:primal_heuristics]{Primal heuristics} +\\
        \hyperref[paragraph:propagation]{Propagation}
      \end{tabular}
    }

    & \hyperref[paragraph:dkp]{$D_k^{+}$} & \DataRow{subtour_dk} \\
    & \hyperref[paragraph:odd_cat]{Odd CAT} & \DataRow{subtour_odd_cat} \\
    & \hyperref[paragraph:odd_path_free]{Odd path-free} & \DataRow{subtour_odd_path_free} \\
    & \hyperref[paragraph:unitary_afc]{Unitary AFC} & \DataRow{subtour_unitary_afc} \\


    \fatline

  \end{tabular*}
\end{table}

\end{center}

\subsection{Learning to branch}
While tailoring the solver lead to significant improvements in the
running times, it is still not possible in a reasonable amount of time to solve 
real world instances to optimality. 

A key problem regarding Branch-and-Bound schemes (and by extension, 
Branch-Price-and-Cut schemes)
is the selection of the branching candidate in the presence of several fractional variables.
To this end, multiple branching rules have been proposed in the literature, 
among these the \emph{strong branching} rule.
Strong branching chooses the ''best'' possible fractional variable with 
respect to a certain score, based on LP-relaxations related to the current 
Branch-and-Bound node (see \cite{StrongBranching} for details).
Strong branching usually yields much smaller Branch-and-Bound
trees. However, the computational costs to evaluate the score of variables is
rather high.
As a result, strong branching is usually only
employed at the root node of the Branch-and-Bound tree.

Khalil et al.~\cite{Khalil2016} suggested to employ machine learning techniques
to learn a branching rule yielding a similar size
of the Branch-and-Bound tree, while avoiding the computational overhead.
The authors used an SVM-based approach to learn a branching rule based
on several generic MIP features, such as fractionality, pseudocost,
and various variable statistics. Labels were assigned based on
strong branching scores. While the results were promising,
the authors were not able to beat the CPLEX-default branching
rule with respect to running time or Branch-and-Bound tree size.

Still, they suggest that the selection and weighing of different
features may be advantageous in order to obtain improved instance-specific
branching models. Furthermore, there has been rapid development regarding
rank learning techniques, mainly driven by web search engine
development (see \cite{LearningToRank} for a summary). As a
result, the SVM-based ranking approach~\cite{SVMRank} has been
superseded by different approaches. Specifically,
the lambdaMART~\cite{Burges2010} algorithm, a boosted tree version of LambdaRank
seems to perform significantly better on web search
related test data.\footnote{A lambdaMART implementation is readily available as part of the
  Quickrank~\cite{Quickrank} \texttt{C++} library.}

In this paper, we use the lambdaMART algorithm to learn a ranking of branching 
candidates depending on some of the features from~\cite{Khalil2016} and some 
features specific to the \tdtsp.  
We generated 20~to~30 training instances. 
For each of these, we collected training data from several branching nodes, 
yielding slightly more than $5000$~data samples to learn the ranking function.
Labels were assigned based on strong branching scores.
The following features were collected for each arc $a=(u,v) \in A$
depending on the value of the current LP-value $Z_{\text{LP}}$ and the cost 
$Z^{*}$ of the best feasible solution available:

\begin{itemize}
\item
  cost relative to the current LP value: $c_{a} / Z_{\text{LP}}$
\item
  cost relative to the best feasible solution: $c_{a} / Z^{*}$
\item
  cost relative to the current gap: $c_{a} / (Z^{*} - Z_{\text{LP}})$
\item distance to one and zero: $x_{uv}$, $1 - x_{uv}$
\item variable slack: $\min(x_{uv}, 1 - x_{uv})$

\item
  number of arcs $(u, v, \theta)$ in $A^{\T}$ divided by $|A^{\T}|$
\item
  number of arcs $(u, v, \theta)$ in $A^{\T}$ which
  have been priced into the current LP
  divided by $|A^{\T}|$
\item
  pseudocost of arc $x_{uv}$
\item
  the (four) sizes relative to $|V|$ of the connected components
  containing $u$ or $v$ of the subgraph containing only the arcs 
  branched to one or zero
\end{itemize}
We trained two different ranking functions; one based on small
instances, each containing $|V| = 10$ vertices, one based
large instances, each containing $|V| = 20$ vertices.
We compared the resulting branching rules with
the one built into SCIP. To this end, we generated
20~small, respectively large, random instances
different from the training instances and
compared the branching rules with respect
to running time and remaining gap.
The results can be found in Table~\ref{table:learning}.

Unfortunately, these first tests were not successful, since the
running times got worse.

\begin{center}
  \newcommand{\DataRow}[1]{
  \csvreader[head to column names,separator=semicolon]{
    Data/Learning/#1.csv}{}{
    & \numSolved &
    \remainingGap &
    \runningTime}}

\newcommand{\mcrot}[4]{\multicolumn{#1}{#2}{\rlap{\rotatebox{#3}{#4}~}}}
\newcommand*\rot{\rotatebox{90}}

\begin{table}[htb]

  \caption{Solving statistics for different branching rules,
    averaged over twenty small and large instances, respectively}
  \label{table:learning}

  \renewcommand{\arraystretch}{1.25}
  \setlength\tabcolsep{4pt}
  \begin{tabular*}{\textwidth}{lp{4cm}rrrrr}

    \fatline

    & & & \mcrot{1}{c}{45}{Number of instances solved} & \mcrot{1}{c}{45}{Remaining gap} & \mcrot{1}{c}{45}{Running time} \\


    \noalign{\vskip 1mm}
    \hline
    \noalign{\vskip 1mm}

    \multirow{3}{*}{
      \begin{tabular}{l}
        Small instances
      \end{tabular}
    }

    & MC-branching, trained on large instances & \DataRow{ranking_branching_large_training_small_instance} \\
    & MC-branching, trained on small instances & \DataRow{ranking_branching_small_training_small_instance} \\
    & Built-in branching & \DataRow{sparse_simple_branching_small_instance} \\

    \noalign{\vskip 1mm}
    \hline
    \noalign{\vskip 1mm}

    \multirow{3}{*}{
      \begin{tabular}{l}
        Large instances
      \end{tabular}
    }

    & MC-branching, trained on large instances & \DataRow{ranking_branching_large_training_large_instance} \\
    & MC-branching, trained on small instances & \DataRow{ranking_branching_small_training_large_instance} \\
    & Built-in branching & \DataRow{sparse_simple_branching_large_instance} \\






    \fatline

  \end{tabular*}
\end{table}

\end{center}

\section{Conclusion}
\label{section:conclusion}

In this paper we have discussed several theoretical and empirical properties
of the \tdtsp. Since the \tdtsp is a generalization of the ATSP, many of
the complexity-specific theoretical results, such as $\NP$-hardness and
inapproximability, carry over to the \tdtsp.

Unfortunately, several positive results regarding the ATSP are not retained in the \tdtsp.
Specifically, the \tdtsp remains inapproximable even if a generalized
triangle inequality is satisfied. Furthermore, even simple
relaxations, such as time-dependent trees cannot be used to
determine combinatorial lower bounds on the \tdtsp.

From a practitioner's point of view, the increase in problem size poses significant
problems when trying to solve even moderate-sized instances of the \tdtsp.
The authors of \cite{TimeDependentTSPPoly} conclude that there are challenging
instance of the STDTSP with less than one hundred vertices. While the results
date back some years, the increase in computational complexity is apparent
even in the case of the STDTSP.

To be able to tackle the \tdtsp, a sophisticated pricing routine is absolutely
necessary. The path-based structure of the formulation is helpful in devising
a pricing routine which employs the technique of Lagrangean relaxations.
The connection between \tdtsp and ATSP yields a variety of feasible classes
of inequalities which help to significantly improve dual bounds. Unfortunately,
the generalizations of STDTSP-type inequalities do not perform equally well in comparison.
Objective value propagation and primal heuristics decrease the gap even further. The
primal heuristics profit from the connection to the ATSP, significantly outperforming
the heuristics built into the solver itself.

Similar to the results in~\cite{Khalil2016}, the learned branching rule 
does not surpass the conventional methods. Still, better features, 
a reinforcement learning approach, or learning parts of a solution directly 
might change the picture.

\bibliographystyle{abbrv}
\bibliography{tdtsp}

\end{document}